\documentclass{amsart}

\newtheorem{proposition}{\textbf{Proposition}}
\newtheorem{theorem}{\textbf{Theorem}}
\newtheorem{corollary}{\textbf{Corollary}}
\newtheorem{remark}{\textbf{Remark}}
\newtheorem{conjecture}{\textbf{Conjecture}}
\newtheorem{lemma}{\textbf{Lemma}}
\newtheorem{question}{\textbf{Question}}

\newtheorem{exe}{\textbf{Example}}
\newcommand{\mem}[1]{\in\mathbb{#1}}
\def\N {\mathbb{N}}
\def\Z {\mathbb{Z}}
\def\Q {\mathbb{Q}}

\def\QQ {\overline{\Q}}
\def\QQQ {\QQ}
\def\C {\mathbb{C}}

\theoremstyle{remark}

\numberwithin{equation}{section}

\usepackage[vmargin=1in,hmargin=1.25in]{geometry}

\begin{document}

\title[SCHANUEL AND ALGEBRAIC POWERS $z^{\,w}$ AND $w^{\,z}$ WITH $z$ AND $w$ TRANSCENDENTAL]{SCHANUEL'S CONJECTURE AND ALGEBRAIC POWERS\\ $\boldsymbol{z^{\,w}}$ AND $\boldsymbol{w^{\,z}}$ WITH $\boldsymbol{z}$ AND $\boldsymbol{w}$ TRANSCENDENTAL}

\author{DIEGO MARQUES}
\address{DEPARTAMENTO DE MATEM\'{A}TICA, UNIVERSIDADE DE BRAS\' ILIA, BRAS\' ILIA, DF, BRAZIL}
\email{diego@mat.unb.br}

\author{JONATHAN SONDOW}
\address{209 WEST 97th STREET, NEW YORK, NY 10025, USA}
\email{jsondow@alumni.princeton.edu}



\begin{abstract}
We give a brief history of transcendental number theory, including Schanuel's conjecture $(S)$. Assuming $(S)$, we prove that if $z$~and~$w$ are complex numbers, not $0$ or $1$, with $z^{\,w}$~and~$w^{\,z}$ algebraic, then $z$~and~$w$ are either both rational or both transcendental. A corollary is that if $(S)$ is true, then we can find transcendental positive real numbers $x$, $y$, and $s\neq t$ such that the three numbers $x^{\,y}\neq y^{\,x}$ and $s^{\,t}=t^{\,s}$ are all integers. Another application (possibly known) is that $(S)$ implies the transcendence of the numbers$$\sqrt{2}^{\,\sqrt{2}^{\,\sqrt{2}}}, \quad i^{\,i^{\,i}},  \quad i^{\,e^{\,\pi}}.$$ We also prove that if $(S)$ holds and  $\alpha^{\,\alpha^{\,z}}=z$, where $\alpha\neq0$ is algebraic and $z$~is irrational, then $z$ is transcendental.
\end{abstract}

\maketitle

\section{Introduction: a brief ``transcendental"\ history} \label{introd}

Recall that a complex number $\alpha$ is \textit{algebraic} if there exists a nonzero polynomial $P\in \Q[x]$ such that $P(\alpha)=0$. (In that case, the smallest degree of such a polynomial is called the \textit{degree} of $\alpha$.) If no such polynomial exists, $\alpha$ is {\it transcendental}. The algebraic numbers form a field $\QQ$, the algebraic closure of the rationals~$\Q$.

Euler was probably the first person to define transcendental numbers in the modern sense (see \cite{erdos}). But transcendental number theory really began in 1844 with Liouville's proof \cite{lio} that if an algebraic number $\alpha$ has degree $n > 1$, then there exists a constant $C>0$ such that $|\alpha-p/q|>Cq^{-n}$, for all $p/q\in \Q\setminus\{0\}$. Using this result, Liouville gave the first explicit examples of transcendental numbers, e.g., Liouville's number $\sum_{n\geq 0}10^{-n!}$.

How large is the set of transcendental numbers? In the 1870s, with his countability arguments, Cantor surprised the mathematical world by proving that almost all complex numbers are transcendental. This motivated the search for such numbers.

In 1872 Hermite \cite{hermite} proved that $e$ is transcendental, and in 1884 Lindemann \cite{lindemann} extended Hermite's method to prove that $\pi$ is also transcendental. In fact, Lindemann proved a more general result.

\begin{theorem} [Hermite-Lindemann] \label{Lindemann}
The number $e^{\alpha}$ is transcendental for any nonzero algebraic number $\alpha$.
\end{theorem}

As a consequence, the numbers $e^{\sqrt{2}}$ and $e^{\,i}$ are transcendental ($i=\sqrt{-1}$), as are $\log 2$ and $\pi$, since $e^{\log 2}=2$ and $e^{\pi i}=-1$ are algebraic. Moreover, the transcendence of $\pi$ resolved the ancient Greek problem of the quadrature of the circle: using straightedge and compass, it is not possible to construct a square and a circle with the same area.

At the 1900 International Congress of Mathematicians in Paris, as the seventh in his famous list of 23 problems, Hilbert gave a big push to transcendental number theory with his question on the arithmetic nature of the power $\alpha^{\,\beta}$ of two algebraic numbers $\alpha$~and~$\beta$. In 1934 Gelfond and Schneider, independently, completely solved the problem (see \cite[p. 9]{baker}).

\begin{theorem}[Gelfond-Schneider] \label{GelSchn}
Assume $\alpha$~and~$\beta$ are algebraic numbers, with $\alpha \neq 0$ or $1$, and $\beta$ irrational. Then $\alpha^{\,\beta}$ is transcendental.
\end{theorem}

In particular, $2^{\sqrt{2}}, (-1)^{\sqrt{2}}$, and $e^{\pi}= i^{\,-2i}$ are all transcendental.

Generalizing the notion of an algebraic number, the complex numbers $\alpha_1,\dots,\alpha_n$ are {\it algebraically dependent} if there exists a nonzero polynomial $P\in \Q[x_1,\dots,x_n]$ such that $P(\alpha_1,\dots,\alpha_n)=0$. Otherwise, $\alpha_1,\dots,\alpha_n$ are {\it algebraically independent}; in particular, they are all transcendental. (More generally, given a subfield $K$ of the complex numbers $\C$, one defines {\it algebraic \emph{(}in\emph{)}dependence over~$K$} by replacing $\Q$ with $K$.)

A major open problem in transcendental number theory is a conjecture of Schanuel which was stated in the 1960s in a course at Yale given by Lang~\cite[pp. 30--31]{La1}.

\begin{conjecture}[Schanuel's conjecture $(S)$] \label{schanuel}
If $\alpha_1,\dots,\alpha_n\in\C$ are linearly independent over~$\Q$, then there are at least $n$ algebraically independent numbers among $\alpha_1,\dots,\alpha_n,e^{\,\alpha_1},\dots,e^{\,\alpha_n}$.
\end{conjecture}

For example, assume $(S)$ and take $\alpha_1=1$ and $\alpha_2=\pi i$. Then at least two of the numbers $1,\ \pi i,\ e,\ -1$ are algebraically independent. Since $1$~and~$-1$ are algebraic, $\pi i$~and~$e$ are algebraically independent. It follows, as $\pi i$~and~$\pi$ are algebraically dependent, that $\pi$~and~$e$ are also algebraically independent. In particular, $\pi + e$ and $\pi e$ are transcendental. (These conclusions are all open problems if $(S)$ is not assumed.)


\section{Statement and applications of the main result} \label{results}

The Gelfond-Schneider Theorem determines the arithmetic nature of $z^{\,w}$ when $z$~and~$w$ are both algebraic (because $z^{\,w}$ is algebraic if $w$ is rational). However, if at least one of the numbers $z$~and~$w$ is transcendental, anything is possible (see Table~\ref{table x,y,x^y}).

\begin{table}[!htb]
\centering
\begin{tabular}{|cc|cc|cc|}
\hline
& \textbf{$z$}\hspace{1.2cm} & & \textbf{$w$} \hspace{2cm} & & \textbf{$z^w$} \hspace{.7cm} \\
\hline
2 & algebraic & $\log 3/\log 2$ & transcendental & 3 & algebraic\\ \hline
2 & algebraic & $i\log 3/\log 2$ & transcendental & $3^{\,i}$ & transcendental\\ \hline
$e^{\,i}$ & transcendental & $\pi$ & transcendental & -1 & algebraic\\ \hline
$e$ & transcendental & $\pi$ & transcendental & $e^{\,\pi}$ & transcendental\\ \hline
$2^{\,\sqrt{2}}$ & transcendental & $\sqrt{2}$ & algebraic & 4 & algebraic\\ \hline
$2^{\,\sqrt{2}}$ & transcendental & $i\sqrt{2}$ & algebraic & $4^{\,i}$ & transcendental \\ \hline
\end{tabular}\newline
\caption{Possibilities for $z^w$ when $z$ or $w$ is transcendental.}
\label{table x,y,x^y}
\end{table}

In his Master's thesis, the first author asked a version of the following question.

\begin{question} \label{T_1^{T_2}}
If $z$~and~$w$ are transcendental, must at least one of the numbers $z^{\,w}$~and~$w^{\,z}$ also be transcendental?
\end{question}

If the answer were yes, then in particular $z^{\,z}$ would be transcendental whenever $z$ is. That would agree with the expected (but still unproved) transcendence of the numbers $e^{\,e},\ \pi^{\,\pi}$, and $(\log 2)^{\,\log 2}$.

In fact, though, \emph{the answer to Question~\ref{T_1^{T_2}} in the case $z$ = $w$ is no}. That was shown by us in~\cite[Proposition~2.2]{SM}, where the Gelfond-Schneider Theorem was used to prove the following.

\begin{proposition} \label{pp}
Given $a\in [e^{-1/e},\infty)$, let $t\in\mathbb{R}^+$ satisfy $t^{\, t}=a$. If either

\emph{(i).} $a\in \Q\setminus\{n^n:n\in \N\}$, or

\emph{(ii).} $a^n\in \QQ \setminus\mathbb{Q}$ for all $n\in \N$,\\
then $t$ is transcendental.
\end{proposition}

For instance, the numbers $t>0$~and~$t_1>0$ which satisfy $t^{\, t}=2$ and $t_1^{\, t_1}=1+\sqrt{2}$ are both transcendental.

Proposition \ref{pp} case~(ii) was generalized by the first author~\cite[Lemma and proof of Proposition]{die30} and was extended  further by him and Jensen~\cite[proof of Theorem~7]{jm}.

Now consider Question~\ref{T_1^{T_2}} in the case $z\neq w$. To study the further subcase when $z^w = w^z$, we recall a classical result (related to a problem posed in 1728 by D.~Bernoulli~\cite[p.~262]{bernoulli}). For a proof, see~\cite[Lemma~3.2]{SM}.

\begin{lemma} \label{xy=yx=z}
Given $r \mem{R}^+$, there exist positive real numbers $s < t$ with $s^{\,t} = t^{\,s} = r$ if and only if $r > e^{\,e} = 15.15426 \ldots$. In that case, $s$ and $t$ are uniquely determined, and $1<s<e<t$.
\end{lemma}

Again using the Gelfond-Schneider Theorem, we proved the following~\cite[Proposition~3.1 and Corollary~3.5]{SM}.

\begin{proposition} \label{(i)RT=TR=A}
Assume the numbers $r,s$, and $t$ are as in Lemma~\ref{xy=yx=z}. If either

\emph{(i).} $16\neq r\in \N$, or

\emph{(ii).} $r^n\in \QQ \setminus\mathbb{Q}$ for all $n\in \N$,\\
then at least one of the numbers $s$~and~$t$ is transcendental.
\end{proposition}

In ~\cite[Conjecture 3.7]{SM} we made the following prediction.
 
\begin{conjecture} \label{conj}
A stronger conclusion holds in Proposition~\ref{(i)RT=TR=A}, namely, that $s$~and~$t$ are both transcendental.
\end{conjecture}

Let us describe the difficulty in proving Conjecture~\ref{conj}. To study the arithmetic nature of the power of two complex numbers, we can use the Gelfond-Schneider Theorem. However, it only applies in the case of algebraic numbers. The nature of $\alpha^{\beta}$, when one or both of the numbers $\alpha$~and~$\beta$ is transcendental, is in general unknown. The sole result in this direction is due to Caveny~\cite{caveny}. He proved that if $\alpha$ is a $T$- or $U$-number and $\beta$ is a $U$-number (as defined also in~\cite[Chapter 10, Section 7H]{paulo}), then $\alpha^{\beta}$ is transcendental. For our problem, Caveny's theorem is not useful, because it seems harder to prove that a complex number is a $T$-number than to prove it is transcendental.

In Section~\ref{proof}, we give a conditional proof of Conjecture~\ref{conj}. In fact, our main result is the following more general one, in which $z^{\,w}$ is not necessarily equal to $w^{\,z}$.

\begin{theorem}\label{main}
Assume Schanuel's conjecture $(S)$ and let $z$~and~$w$ be complex numbers, not $0$ or~$1$. If $z^{\,w}$~and~$w^{\,z}$ are algebraic, then $z$~and~$w$ are either both rational or both transcendental. In particular, if $(S)$ is true, then Conjecture~\ref{conj} is also true.
\end{theorem}

Here is an application.

\begin{corollary} \label{COR: z=/=w}
Assuming $(S)$, we can find transcendental positive real numbers $x,\, y$, and $s\neq t$ such that the three numbers $x^{\,y}\neq y^{\,x}$ and $s^{\,t}=t^{\,s}$ are integers.
\end{corollary}
\begin{proof}
Define $f(X)=X^{\,2^{\,1/X}}$. Then $f(2)<3<f(3)$, and so by continuity we can choose $x \in (2,3)$ with $f(x)=3$. Set $y =2^{\,1/{\,x}}$. From the equalities $x^{\,y}=3$ and $y^{\,x}=2$, we deduce that $x$~and~$y$ cannot both be rational. If $(S)$ is true, then by Theorem~\ref{main} both $x$~and~$y$ are transcendental.

By Lemma~\ref{xy=yx=z}, we can find $s\neq t$ in $\mathbb{R}^+$ such that $s^t=t^s=17$. If $(S)$ is true, then Proposition~\ref{(i)RT=TR=A} and Theorem~\ref{main} imply that $s$~and~$t$ are both transcendental.
\end{proof}

Summarizing, \emph{the answer to Question~\ref{T_1^{T_2}} is unconditionally no in the case $z = w$, and is conditionally no in the case $z \neq w$ both when $z^w \neq w^z$ and when $z^w = w^z$.}

Here is another consequence of Theorem~\ref{main}.

\begin{corollary} \label{COR: sqrt(2)}
Assume $(S)$ and let $\alpha,\beta,\gamma$ be nonzero complex numbers, with $\alpha\neq1$ and $\beta^{\,\gamma}\neq1$. Suppose that at least one of $\alpha$ and $\beta^{\,\gamma}$ is irrational, and at least one is algebraic. If $\beta^{\,\gamma\,\alpha}$ is also algebraic, then $\alpha^{\,\beta^{\,\gamma}}$ is transcendental.
\end{corollary}
\begin{proof}
Set $z=\alpha$ and $w=\beta^{\,\gamma}$. Since $w^{\,z}=\beta^{\,\gamma\,\alpha}$ is algebraic, Theorem~\ref{main} implies that $z^{\,w}=\alpha^{\,\beta^{\,\gamma}}$ must be transcendental.
\end{proof}

The following examples of Corollary \ref{COR: sqrt(2)} may be known, but we have not found them in the literature.

\begin{exe} \label{EX: sqrt(2)}
\emph{Conjecture $(S)$~implies the transcendence of the numbers
\begin{equation}
    \sqrt{2}^{\,\sqrt{2}^{\,\sqrt{2}}}, \quad i^{\,i^{\,i}},  \quad i^{\,e^{\,\pi}}, \label{EQ: a^b^c}
\end{equation}
because each of them is of the form $\alpha^{\,\beta^{\,\gamma}}$, where $\alpha\in\QQ\setminus\Q$ and $\beta^{\,\gamma}\neq1$ and $\beta^{\,\gamma\,\alpha}\mem{\QQ}$.}
\end{exe}

We now give another application of Schanuel's Conjecture. Notice first that, by the Gelfand-Schneider Theorem, if $\alpha^{\,z}=z$, where $\alpha\neq0$ is algebraic and $z$~is irrational, then $z$ is transcendental. (For instance, from the example $t^{\,t} =2$ for Proposition~\ref{pp}, we get $(1/2)^{\,z}=z$, where $z=1/t\not\mem{\QQ}$.) The following statement is stronger, since $\alpha^{\,z}=z$ implies $\alpha^{\,\alpha^{\,z}}=z$, but not conversely (see~\cite[Section~4]{SM}).

\begin{conjecture} \label{conj2}
Let $\alpha\neq0$~and~$z$ be complex numbers, with $\alpha$ algebraic and $z$~irrational. If $\alpha^{\,\alpha^{\,z}}=z,$ then $z$ is transcendental.
\end{conjecture}

\begin{theorem}\label{main2}
If $(S)$ is true, then Conjecture~\ref{conj2} is also true.
\end{theorem}

Theorem \ref{main2} yields a conditional proof of~\cite[Conjecture 4.6]{SM}, because the latter is a consequence of Conjecture~\ref{conj2}.

The proofs of Theorems~\ref{main} and~\ref{main2} are given in Section~\ref{proof}.


\section{Preliminaries on Schanuel's conjecture $(S)$} \label{SEC: schanuel}

Here are two more consequences of $(S)$ if it is true. (They, too, are open problems if $(S)$ is not assumed.)

\begin{itemize}
\item \emph{The numbers $e,e^{\pi},e^e,e^i,\pi,\pi^e,\pi^{\pi},\pi^i,2^{\pi},2^e,2^i,\log \pi,\log 2, \log 3, \log \log 2,$\\ $(\log 2)^{\log 3}$, and $2^{\sqrt{2}}$ are algebraically independent. In particular, they are all transcendental.} (The proof in~\cite[Conjecture $(S_7)$, p. 326]{paulo}, like our proof of Theorem~\ref{main}, invokes $(S)$ four times.)
\item \emph{The numbers $e,e^e, e^{e^e},\dots$ are algebraically independent.} (Take $\alpha_1=1$ and $\alpha_2=e$ and proceed by induction.)
\end{itemize}

Also, $(S)$ implies generalizations of many important theorems in transcendental number theory. We mention two. Proofs of them and of several other classical consequences of $(S)$, together with an elegant exposition of it, can be found in \cite[Chapter 10, Section 7G]{paulo}. See also~\cite{chow}.

\begin{itemize}
\item (Generalization of the Hermite-Lindemann Theorem) \emph{If $\alpha$ is a nonzero algebraic number and $(S)$ is true, then
$$
e^{e^{\,\cdot^{\cdot^{\cdot^{e^{\alpha}}}}}}
$$
is transcendental.}
\item (Generalization of the Gelfond-Schneider Theorem) \emph{If $(S)$ is true and $\alpha$~and~$\beta$ are algebraic numbers, with $\alpha \neq 0$ or $1$, and $\beta$ irrational, then $\alpha^{\,\beta}$ and $\log \alpha$ are algebraically independent.}

\end{itemize}

In addition, $(S)$ implies the algebraic independence of certain numbers over fields different from~$\Q$:
\begin{itemize}
\item \emph{Set $E=\cup_{n=0}^{\infty} E_n$, where $E_0=\QQ$ and $E_n=\overline{E_{n-1}(\{e^{\alpha}:\alpha \in E_{n-1}\})}$, for $n\geq 1$. If $(S)$ is true, then the numbers $$
\pi, \log \pi, \log \log \pi,\log \log \log \pi, \dots
$$
are algebraically independent over $E$}. (For the proof, see \cite{die2}.)
\end{itemize}

We now recall some definitions.

Given field extensions $\C\supset L\supset K$, a subset $B$ of $L$ is a \textit{transcendence basis} of $L$ over $K$ if the elements of $B$ are algebraically independent over $K$ and if furthermore $L$ is an algebraic extension of the field $K(B)$ (the field obtained from $K$ by adjoining the elements of $B$). One can show that every field extension $L/K$ has a transcendence basis $B\subset L$, and that all transcendence bases have the same cardinality $|B|$. This cardinality is the \textit{transcendence degree} of the extension, and is denoted trdeg$_K L$ or trdeg$(L/K)$.

With this definition, $(S)$ can be restated: \textit{If $\alpha_1,\ldots,\alpha_n$ are linearly independent over $\Q$, then}
\begin{center}
    trdeg$_{\Q}\Q(\alpha_1,\dots,\alpha_n,e^{\alpha_1},\dots,e^{\alpha_n})\geq n.$
\end{center}

Here are some facts about transcendence degree that we shall use in the next section. (For proofs, see~\cite[Chapter VIII]{La2}.) Let $X$~and~$Y$ be finite subsets of $\C$.
\begin{itemize}
\item[(i).] If $X\subset \QQQ$, then trdeg$_\Q\Q(X\cup Y)=$ trdeg$_\Q\Q(Y)$. (Algebraic numbers do not contribute to the transcendence degree.)
\item[(ii).] If $X\subset Y$, then trdeg$_\Q\Q(X \cup Y)=$ trdeg$_\Q\Q(Y)$. (Only distinct numbers can contribute to the transcendence degree.)
\item[(iii).] If trdeg$_\Q\Q(Y)=|Y|$, then $Y$ is an algebraically independent set.
\item[(iv).] We have trdeg$_\Q\Q(X)=$ trdeg$_{\QQ}\QQ(X)$. (It makes no difference to say that a set is algebraically independent over $\Q$ or over $\QQ$.)
\end{itemize}

\indent
Recall also that a set of nonzero complex numbers is \textit{multiplicatively independent} if for any finite subset $\{x_1,\dots,x_m\}$ the relation $x_1^{a_1}\cdots x_m^{a_m}=1$, with integer exponents $a_1,\dots,a_m$, implies $a_1=\cdots=a_m=0$. Otherwise, the set is {\it multiplicatively dependent}. For example, the set of prime numbers is multiplicatively independent (by the Fundamental Theorem of Arithmetic). Any algebraically independent set is also multiplicatively independent. 

\begin{remark}\label{log}
\emph{It is a simple matter to show that $x_1,\dots,x_m$ are multiplicatively independent if and only if $\log x_1,\dots,\log x_m$ are linearly independent over $\Q$. It follows easily that $(S)$ can be restated:} If $\alpha_1,\ldots,\alpha_n$ are multiplicatively independent, then
\begin{center}
    \emph{trdeg}$_{\Q}\Q(\alpha_1,\dots,\alpha_n,\log \alpha_1,\dots,\log \alpha_n)\geq n.$
\end{center}
\end{remark}

\noindent This is the form of $(S)$ we shall use in the proof of Theorem~\ref{main}.


\section{Proofs of Theorems~\ref{main} and~\ref{main2}} \label{proof}

\noindent\emph{Proof of Theorem~\ref{main}}. Assume that $z$~and~$w$ are not equal to $0$ or $1$, and that$$\alpha:=z^w, \qquad \beta:=w^z$$are algebraic. By the Gelfond-Schneider Theorem, to prove Theorem~\ref{main} it suffices to show that if $w$ is transcendental, then so is $z$. Suppose on the contrary that $z$ is algebraic. Since $z\neq0$ or $1$, and $w\neq0$, taking logarithms gives
\begin{equation} \label{eq1}
    w = \dfrac{\log \alpha}{\log z}, \qquad \log w = \dfrac{\log \beta}{z}.
\end{equation}
We first show that $\alpha,z,\log \alpha, \log z$ are multiplicatively independent.

Since $\log \alpha/ \log z=w$ is transcendental, $\alpha$~and~$z$ are multiplicatively independent. Then $(S)$ implies trdeg$_\Q\Q(\alpha, z,\log \alpha,\log z)\geq 2$. Hence, as $\alpha$~and~$z$ are algebraic, $\log \alpha$ and $\log z$ are algebraically independent (see facts~(i) and (iii)). It follows that any solution $(a,b,c,d)\in \Z^4$ to the equation
$$
\alpha^az^b(\log \alpha)^c(\log z)^d=1
$$
has $c=d=0$ (see~(iv)). Thus $\alpha^az^b=1$, and now the multiplicative independence of $\alpha$~and~$z$ implies $a=b=0$. Thus $a=b=c=d=0$, and so the numbers $\alpha,z,\log \alpha, \log z$ are multiplicatively independent.

Again by $(S)$, 
\begin{center}
trdeg$_{\Q}\Q(\alpha,z,\log \alpha,\log z,\log \alpha,\log z,\log \log \alpha, \log \log z)\geq 4,$
\end{center}
and as $\alpha$~and~$z$ are algebraic, $\log \alpha,\log z,\log \log \alpha, \log \log z$ are algebraically independent (see (i), (ii), and (iii)). 

We now prove that $\alpha,\beta,z$ are multiplicatively independent. Say \mbox{$\alpha^a\beta^bz^c=1$}, where \mbox{$(a,b,c)\in \Z^3$.} If $b\neq 0$, then taking logarithms gives $\log \beta\in \Q(\log \alpha,\log z)$. But that implies, since taking logarithms in the first equality in \eqref{eq1} and substituting the second yields the relation
\begin{equation}\label{key}
    z^{-1}\log \beta=\log \log \alpha-\log \log z,
\end{equation}
that $\log \log \alpha- \log \log z \in \Q(z,\log \alpha,\log z)$, contradicting the algebraic independence of $\log \alpha,\log z,$ $\log \log \alpha, \log \log z$. Thus $b=0$. Now $a=c=0$, because $\alpha$~and~$z$ are multiplicatively independent. Hence $\alpha,\beta,z$ are multiplicatively independent.

Then by $(S)$
\begin{center}
    trdeg$_{\Q}\Q(\alpha,\beta,z,\log \alpha,\log \beta,\log z)\geq 3,$
\end{center}
and so, as $\alpha,\beta,z$ are algebraic, $\log \alpha,\log \beta,\log z$ are algebraically independent. It follows that any solution $(a,b,c,d,f,g)\in \Z^6$ to the equation
$$
\alpha^a\beta^bz^c(\log \alpha)^d(\log \beta)^f(\log z)^g=1
$$
has $d=f=g=0$. Then $a=b=c=0$, because $\alpha,\beta,z$ are multiplicatively independent. We conclude that $\alpha,\beta,z,\log \alpha,\log \beta,\log z$ are multiplicatively independent. Now $(S)$ implies
$$
\mbox{trdeg}_{\Q}\Q(\alpha,\beta,z,\log \alpha,\log \beta,\log z,\log \alpha,\log \beta, \log z,\log \log \alpha,\log \log \beta, \log \log z) \ge 6.
$$
Hence, as $\alpha,\beta,z$ are algebraic, $\log \alpha, \log \beta, \log z,\log \log \alpha,\log \log \beta,\log \log z$ are algebraically independent. But that contradicts the relation \eqref{key}. Therefore, $z$ is transcendental. This completes the proof of Theorem~\ref{main}.
\qed\\

\noindent\emph{Proof of Theorem~\ref{main2}}. Suppose on the contrary that $(S)$ holds, and that $\alpha\neq0$ and $z\not\mem{Q}$ are both algebraic, and $\alpha^{\,\alpha^{\,z}}=z$. Then $\alpha\neq1$, and so, by the Gelfond-Schneider Theorem, $\alpha^{\,z}\not\mem{\QQ}$. If $a+bz+c\,\alpha^{\,z}=0$, where $(a,b,c)\mem{Z}^3$, then $a+bz\mem{\QQ}$ and $\alpha^{\,z}\not\mem{\QQ}$ imply $c=0$. Now $z\not\mem{Q}$ gives $a=b=0$. Thus $1,z,\alpha^{\,z}$ are linearly independent over~$\Q$. Multiplying them by $\log\alpha\neq0$, we get that $\log\alpha,z\log\alpha,\alpha^{\,z}\log\alpha$ are also linearly independent over~$\Q$. Hence by $(S)$
\begin{equation*}
    \tau := \mbox{trdeg}_{\Q}\Q\left(\log\alpha,z\log\alpha,\alpha^{\,z}\log\alpha,\alpha,\alpha^{\,z},\alpha^{\,\alpha^{\,z}}\right)\ge3.
\end{equation*}
But since $\alpha\mem{\QQ}$ and $\alpha^{\,\alpha^{\,z}}=z\mem{\QQ}$, it follows that $\tau = \mbox{trdeg}_{\Q}\Q(\log\alpha,\alpha^{\,z})\le2$ (see facts (i), (ii), and (iv)), a~contradiction. This proves the theorem.
\qed

\section*{Acknowledgement}

The first author is grateful to FEMAT for financial support.

\end{document}